\documentclass[a4paper,12pt]{amsart}
\usepackage{amsfonts,stmaryrd,amssymb,amsmath,a4wide,verbatim,url,epsfig,color,hyperref,mathrsfs,graphicx,pgf,tikz,amsthm,mathtext,cite,enumerate,float}
\usepackage[british]{babel}
\usepackage[active]{srcltx}
\usepackage[T1,T2A]{fontenc}
\usepackage[utf8]{inputenc}
\usetikzlibrary{arrows}
\usepackage[matrix,arrow,curve]{xy}

\makeatletter
\@addtoreset{equation}{section}\makeatother

\newtheorem{conj}{Conjecture}[section]
\newtheorem{theorem}{Theorem}[section]
\theoremstyle{definition}
\newtheorem{rem}{Remark}[section]
\theoremstyle{definition}
\newtheorem{lemma}{Lemma}[section]
\theoremstyle{definition}

\begin{document}

\title[Proof of a Conjecture of Wiegold for nilpotent Lie algebras]{Proof of a Conjecture of Wiegold for nilpotent Lie algebras}



\author[Alexander Skutin]{Alexander Skutin}

\maketitle

\section{Introduction}

In this short note we confirm an analog of a conjecture of James Wiegold \cite[4.69]{K} for finite dimensional nilpotent Lie algebras.

\begin{conj}[Wiegold's conjecture]

Let $G$ be a finite $p$-group and let $|G'|>p^{n(n-1)/2}$ for some non-negative integer $n$. Then the group $G$ can be generated by the elements of breadth at least $n$. The breadth $b(x)$ of an element $x$ of a finite $p$-group $G$ is defined by the equation $|G:C_G(x)| = p^{b(x)}$, where $C_G(x)$ is the centralizer of $x$ in $G$.

\end{conj}

An overview of this problem can be found in \cite{W}. In \cite{Br} M.R.Vaughan-Lee proved that for a finite $p$-group $G$ of breadth $b = \max_{g\in G}b(g)$, we have $|G'|\leq p^{b(b-1)/2}$. He also showed that for a finite dimensional nilpotent Lie algebra $\mathfrak{g}$ of breadth $b = \max_{g\in \mathfrak{g}}b(g)$, $\dim[\mathfrak{g},\mathfrak{g}]\leq b(b + 1)/2$.

Conjecture 1.1 was proved by the Author in \cite{Sk}. The goal of this paper is to prove the Lie algebra analog of this conjecture.

\medskip

The breadth $b(x)$ of an element $x$ of a finite dimensional Lie algebra $\mathfrak{g}$ over a field $\mathbb{F}$ is defined by the equation $\dim \mathfrak{g} - \dim C_{\mathfrak{g}}(x) = b(x)$, where $C_{\mathfrak{g}}(x)$ is the Lie centralizer of $x$ in $\mathfrak{g}$. We prove the following

\begin{conj}[Wiegold's conjecture for nilpotent Lie algebras]

Let $\mathfrak{g}$ be a nilpotent Lie algebra over a field $\mathbb{F}$ and let $\dim\mathfrak{g}'>n(n-1)/2$ for some non-negative integer $n$. Then the Lie algebra $\mathfrak{g}$ can be generated by the elements of breadth at least $n$.

\end{conj}

In this article we confirm the Conjecture 1.2 and also show that the more general results are true:

\begin{theorem}

Let $\mathfrak{g}$ be a nilpotent Lie algebra over an infinite field $\mathbb{F}$ and let $\dim\mathfrak{g}'>n(n-1)/2$ for a non-negative $n$. Then the set of elements of breadth at least $n$ cannot be covered by a finite number of proper Lie subalgebras in $\mathfrak{g}$.

\end{theorem}

\begin{theorem}

Let $\mathfrak{g}$ be a nilpotent Lie algebra over a finite field $\mathbb{F}\not=\mathbb{F}_2$ and let $\dim\mathfrak{g}'>n(n-1)/2$ for some non-negative integer $n$. Then the set of elements of breadth at least $n$ cannot be covered by $|\mathbb{F}|-1$ proper Lie subalgebras in $\mathfrak{g}$.

\end{theorem}

In the particular case $\mathbb{F} = \mathbb{F}_2$ we also get a more general result:

\begin{theorem}

Let $\mathfrak{g}$ be a nilpotent Lie algebra over a field $\mathbb{F}_2$ and let $\dim\mathfrak{g}'>n(n-1)/2$ for some non-negative integer $n$. Then the set of elements of breadth at least $n$ cannot be covered by two proper Lie subalgebras in $\mathfrak{g}$, one of which has codimension at least 2 in $\mathfrak{g}$.

\end{theorem}

\section{Formulations and proofs of the main Lemmas}

The breadth $b_\mathfrak{h}(x)$ of an element $x$ of a finite dimensional Lie algebra $\mathfrak{g}$ with respect to a subalgebra $\mathfrak{h}\subseteq \mathfrak{g}$ is defined by the equation $\dim\mathfrak{h} - \dim C_\mathfrak{h}(x) = b_\mathfrak{h}(x)$, where $C_\mathfrak{h}(x) = \{h\in\mathfrak{h}| [x, h]=0\}$ is the centralizer of $x$ in $\mathfrak{h}$. By definition, $b(x) = b_\mathfrak{g}(x)$.


The following two lemmas are the well-known facts in theory of Lie algebras, so we state them without proof.

\begin{lemma}

Let $\mathfrak{g}$ be a Lie algebra over a finite field $\mathbb{F}$. Then $\mathfrak{g}$ cannot be covered by $|\mathbb{F}|$ proper subalgebras. Additionally, if $\mathfrak{g}$ is covered by $|\mathbb{F}|+1$ proper subalgebras $\mathfrak{h}_1$, $\mathfrak{h}_2$, $\ldots$, $\mathfrak{h}_{|\mathbb{F}|+1}$, then each of these $\mathfrak{h}_i$ must have codimension $1$ in $\mathfrak{g}$.

\end{lemma}

\begin{lemma}

Let $\mathfrak{g}$ be a nilpotent Lie algebra over a field $\mathbb{F}$, such that there exists a central subalgebra $\mathfrak{f}$ of codimension $2$ in $\mathfrak{g}$. Then $\dim\mathfrak{g}'\leq 1$.

\end{lemma}

\begin{lemma}

Consider an ideal $\mathfrak{h}$ of a finite dimensional nilpotent Lie algebra $\mathfrak{g}$. Let $\mathfrak{f}$ be the ideal of $\mathfrak{g}$ generated by elements $g$ from $\mathfrak{h}$, such that $b_{\mathfrak{h}}(g) = b(g)$. Then

\begin{enumerate}
    \item If $\mathfrak{f} = \mathfrak{h}$, then $\mathfrak{g}' = \mathfrak{h}'$;
    \item If $\mathfrak{f}$ has a codimension at most $1$ in $\mathfrak{h}$, then $\mathfrak{h}'$ has a codimension at most $1$ in $\mathfrak{g}'$.
\end{enumerate}

\end{lemma}

\begin{proof}

Consider the factorization homomorphism $\pi:\mathfrak{g}\to\mathfrak{g}/\mathfrak{h}'$. For each element $g$ in $\mathfrak{h}$, such that $b_{\mathfrak{h}}(g) = b(g)$ we have $\{[g,g']|g'\in\mathfrak{g}\} = \{[g,h]|h\in\mathfrak{h}\}\subseteq\mathfrak{h}'$, so $\pi(g)$ lies in the center of $\pi(\mathfrak{g})$ and $\pi(\mathfrak{f})$ is the central Lie subalgebra of $\pi(\mathfrak{g})$. If $\mathfrak{f} = \mathfrak{h}$, then $\pi(\mathfrak{f}) = \pi(\mathfrak{h})$ is the central Lie subalgebra of $\pi(\mathfrak{g})$ of codimension $1$, so $\pi(\mathfrak{g})$ is abelian and $\mathfrak{g}' = \mathfrak{h}'$. If $\mathfrak{f}$ has a codimension at most $1$ in $\mathfrak{h}$, then $\pi(\mathfrak{f})$ is the central Lie subalgebra of $\pi(\mathfrak{g})$ of codimension at most $2$, so from Lemma 2.2 we get $\dim\mathfrak{g}'\leq 1$ and $\mathfrak{h}'$ has a codimension at most $1$ in $\mathfrak{g}'$.

\end{proof}

\begin{lemma}

Let $\mathfrak{g}$ be a finite dimensional Lie algebra and let $\mathfrak{h}$ be its ideal of codimension $1$ in $\mathfrak{g}$. Then for any element $x$ from the set $\mathfrak{g}\setminus\mathfrak{h}$ we have $\dim\mathfrak{g}'\leq b(x) + \dim\mathfrak{h}'$.

\end{lemma}

\begin{proof}

The set $V=\{[x, h]| h\in\mathfrak{h}\}$ form a vector space with the dimension not bigger than $b(x)$. So it is enough to prove that $\mathfrak{g}' = V+\mathfrak{h}'$. This follows from the following properties of the set $V+\mathfrak{h}'$ :

\begin{enumerate}
    \item $V+\mathfrak{h}'$ is a vector space;
    \item $V+\mathfrak{h}'$ is an ideal in $\mathfrak{g}$ : $[[x, h], g]+\mathfrak{h}' = \overbrace{[x, \underbrace{[h, g]}_{\in\mathfrak{h}}]}^{\in V}+\overbrace{[h, \underbrace{[x, g]}_{\in\mathfrak{h}}]}^{\in\mathfrak{h}'}+\mathfrak{h}'$;
    \item Lie algebra $\mathfrak{g}/(V+\mathfrak{h}')$ is abelian.
\end{enumerate}
These facts imply the Lemma 2.4.

\end{proof}

\section{Proof of Theorem 1.1}

Assume the converse. Let the proper subalgebras $\mathfrak{h}_1$, $\mathfrak{h}_2$, $\ldots$, $\mathfrak{h}_k$ cover all the elements of breadth at least $n$. We will prove that $\dim\mathfrak{g}'\leq n(n-1)/2$. Considering a sufficiently large finite-dimensional subalgebra in $\mathfrak{g}$, we can assume that $\mathfrak{g}$ is finite-dimensional. The proof is by induction on $\dim\mathfrak{g}$. We can assume that $\mathfrak{h}_i$ are the maiximal ideals of codimension $1$ in $\mathfrak{g}$ (because every proper subalgebra in the nilpotent and finite dimensional Lie algebra is contained in the ideal of codimension $1$). Consider any ideal $I\not=\mathfrak{h}_i$ of codimension $1$ in $\mathfrak{g}$. Denote by $J$ the ideal of $\mathfrak{g}$ generated by elements $g$ from $I$ such that $b_I(g)=b(g)$. Consider the case when $J = I$. Applying the Lemma 2.3 we conclude that $I' = \mathfrak{g}'$. The rest follows from the induction hypothesis : the Lie algebra $I$ has smaller dimension and all its elements of breadth at least $n$ are contained in the union of proper subalgebras $I\cap \mathfrak{h}_i$ (because $b_I(x)\leq b(x)$). Now consider the case $I\not= J$. Notice that the set $I\setminus \left(J\cup_i\mathfrak{h}_i\right)$ contain only the elements of breadth at most $n-2$ (because $J$ is generated by elements $\{g\in I| b_I(g)=b(g)\}$). Thus, from the induction hypothesis applied to $I$ and its $k+1$ proper subalgebras $J$, $I\cap\mathfrak{h}_1$, $I\cap\mathfrak{h}_2$, $\ldots$, $I\cap\mathfrak{h}_k$ we conclude that $\dim I'\leq (n-1)(n-2)/2$. Finally, consider any element $a$ not lying in $I\cup_i\mathfrak{h}_i$, its breadth is less than $n$, so from the Lemma 2.4 we get $\dim\mathfrak{g}'\leq b(a) + \dim I'\leq\frac{n(n-1)}{2}$.

\hfill$\Box$

\section{Proof of Theorem 1.2}

Assume the converse. Let the proper subalgebras $\mathfrak{h}_1$, $\mathfrak{h}_2$, $\ldots$, $\mathfrak{h}_{|\mathbb{F}|-1}$ cover all the elements of breadth at least $n$. We will prove that $\dim\mathfrak{g}'\leq n(n-1)/2$. Considering a sufficiently large finite-dimensional subalgebra in $\mathfrak{g}$, we can assume that $\mathfrak{g}$ is finite-dimensional. The proof is by induction on $\dim\mathfrak{g}$. We can assume that $\mathfrak{h}_i$ are the maiximal ideals of codimension $1$ in $\mathfrak{g}$ (because every proper subalgebra in the nilpotent and finite dimensional Lie algebra is contained in the ideal of codimension $1$) and $\mathfrak{h}_1\not= \mathfrak{h}_2$ (because in every non-abelian finite dimensional and nilpotent Lie algebra $\mathfrak{g}$, there are at least two different maximal proper subalgebras). Consider any ideal $I\not=\mathfrak{h}_i$ of codimension $1$ in $\mathfrak{g}$ such that $I\cap \mathfrak{h}_1 = I\cap \mathfrak{h}_2 = \mathfrak{h}_1\cap \mathfrak{h}_2$. Denote by $J$ the ideal of $\mathfrak{g}$ generated by elements $g$ from $I$ such that $b_I(g)=b(g)$. Consider the case when $J = I$. Applying the Lemma 2.3 we conclude that $I' = \mathfrak{g}'$. The rest follows from the induction hypothesis : the Lie algebra $I$ has smaller dimension and all its elements of breadth at least $n$ are contained in the union of proper subalgebras $I\cap \mathfrak{h}_i$ (because $b_I(x)\leq b(x)$). Now consider the case $I\not= J$. Notice that the set $I\setminus \left(J\cup_i\mathfrak{h}_i\right)$ contain only the elements of breadth at most $n-2$ (because $J$ is generated by elements $\{g\in I| b_I(g)=b(g)\}$). Thus, from the induction hypothesis applied to $I$ and its $|\mathbb{F}|-1$ proper subalgebras $J$, $I\cap\mathfrak{h}_1 = I\cap\mathfrak{h}_2$, $ I\cap\mathfrak{h}_3$ $\ldots$, $I\cap\mathfrak{h}_{|\mathbb{F}|-1}$ we conclude that $\dim I'\leq (n-1)(n-2)/2$. Finally, consider any element $a$ not lying in $I\cup_i\mathfrak{h}_i$ (such element exists because of the Lemma 2.1), its breadth is less than $n$, so from the Lemma 2.4 we get $\dim\mathfrak{g}'\leq b(a) + \dim I'\leq\frac{n(n-1)}{2}$. $\Box$

\section{Proof of Theorem 1.3}

\begin{theorem}

Let $\mathfrak{g}$ be a nilpotent Lie algebra over a finite field $\mathbb{F}$. Let the next two conditions hold for some integers $n\leq k + 1$

\begin{enumerate}

\item The set of all elements of the breadth at least $n$ can be covered by $|\mathbb{F}|$ subalgebras of $\mathfrak{g}$;

\item The set of elements of the breadth at most $k$ generates $\mathfrak{g}$.

\end{enumerate}

Then $\dim\mathfrak{g}'\leq\frac{(n-1)(n-2)}{2} + k$.

\end{theorem}

\begin{proof}

Considering a sufficiently large finite-dimensional subalgebra in $\mathfrak{g}$, we can assume that $\mathfrak{g}$ is finite-dimensional. The proof is by induction on $\dim\mathfrak{g}$. Let the proper subalgebras $\mathfrak{h}_1$, $\mathfrak{h}_2$, $\ldots$, $\mathfrak{h}_{|\mathbb{F}|}$ cover all the elements of breadth at least $n$. We can assume that $\mathfrak{h}_i$ are the maiximal ideals of codimension $1$ in $\mathfrak{g}$ (because every proper subalgebra in the finite dimensional and nilpotent Lie algebra is contained in the ideal of codimension $1$) and $\mathfrak{h}_1\not= \mathfrak{h}_2$ (because in every non-abelian finite dimensional and nilpotent Lie algebra $\mathfrak{g}$, there are at least two different maximal proper subalgebras). Consider any ideal $I\not=\mathfrak{h}_i$ of codimension $1$ in $\mathfrak{g}$ such that $I\cap \mathfrak{h}_1 = I\cap \mathfrak{h}_2 = \mathfrak{h}_1\cap \mathfrak{h}_2$. Denote by $J$ the ideal of $\mathfrak{g}$ generated by elements $g$ from $I$, such that $b_I(g)=b(g)$. Consider the case when $J = I$. Applying the Lemma 2.3 we conclude that $I' = \mathfrak{g}'$. The rest follows from the induction hypothesis : the Lie algebra $I$ has smaller dimension and all its elements of breadth at least $n$ are contained in the union of proper subalgebras $I\cap \mathfrak{h}_i$ (because $b_I(x)\leq b(x)$), also the set $I\setminus\cup_{i\geq 2}\mathfrak{h}_i$ generates $I$ (because of the Lemma 2.1) and consists only of elements of breadth at most $n-1\leq k$.

So we can assume that $J$ is a proper ideal of $I$. Consider the case when the codimension of $J$ in $I$ is $1$. Apply the Lemma 2.3 to $I$, we conclude that $\dim\mathfrak{g}'\leq\dim I'+1$. Apply the induction hypothesis to the Lie algebra $I$ and to its proper subalgebras $J$, $\mathfrak{h}_1\cap\mathfrak{h}_2$, $\mathfrak{h}_3$, $\ldots$, $\mathfrak{h}_{|\mathbb{F}|}$ so we get $\dim I'\leq\frac{(n-3)(n-2)}{2} + n - 1$ (every element from the set $I\setminus\left((\mathfrak{h}_1\cap\mathfrak{h}_2)\cup\mathfrak{h}_3\cup\ldots\cup\mathfrak{h}_{|\mathbb{F}|}\right)$ is of the breadth at most $n-1$ in $I$ and this set generates $I$ (because of the Lemma 2.1), also for any element $g$ from the set $I\setminus\left(J\cup(\mathfrak{h}_1\cap\mathfrak{h}_2)\cup\mathfrak{h}_3\cup\ldots\cup\mathfrak{h}_{|\mathbb{F}|}\right)$, we have $b_I(g)\leq n-2$). So if $k\geq 2$ we get $\dim \mathfrak{g}'\leq\dim I' + 1 \leq \frac{(n-3)(n-2)}{2}+n=\frac{(n-2)(n-1)}{2}+2\leq\frac{(n-2)(n-1)}{2}+k$ and the induction step is clear. In the case $k\leq 1$, we get $n=2$ (the case $n\leq 1$ is trivial) and the set $I\setminus\left(J\cup(\mathfrak{h}_1\cap\mathfrak{h}_2)\cup\mathfrak{h}_3\cup\ldots\cup\mathfrak{h}_{|\mathbb{F}|}\right)$ is contained in the center of $I$. It is clear that the set $I\setminus\left(J\cup(\mathfrak{h}_1\cap\mathfrak{h}_2)\cup\mathfrak{h}_3\cup\ldots\cup\mathfrak{h}_{|\mathbb{F}|}\right)$ generates the central subalgebra of the codimension at most $1$ in $I$ (Lemma 2.1), so $I$ is abelian. From the conditions of Theorem 5.1 there exists an element $g\notin I$ such that $b(g)\leq k\leq 1$. So the Lie algebra $C_\mathfrak{g}(g)$ has codimension at most $1$ in $\mathfrak{g}$. Also the subalgebra $C_\mathfrak{g}(g)\cap I$ has codimension at most $2$ in $\mathfrak{g}$ and is central in $\mathfrak{g}$. Apply Lemma 2.2 to the Lie algebra $\mathfrak{g}$, so we conclude that $\dim\mathfrak{g}'\leq 1\leq\frac{(n-1)(n-2)}{2} + k$.

Eventually, we can assume that the codimension of $J$ in $I$ is greater or equal to $2$. So the set $I\setminus\left(J\cup(\mathfrak{h}_1\cap\mathfrak{h}_2)\cup\mathfrak{h}_3\cup\ldots\cup\mathfrak{h}_{|\mathbb{F}|}\right)$ generates $I$ (Lemma 2.1), thus, $I$ is generated by the elements of the breadth at most $n-2$ in $I$. And we can apply the induction hypothesis to $I$ and its subalgebras $J$, $\mathfrak{h}_1\cap\mathfrak{h}_2$, $\mathfrak{h}_3$, $\ldots$, $\mathfrak{h}_{|\mathbb{F}|}$ and conclude that $\dim I'\leq\frac{(n-3)(n-2)}{2}+n-2 = \frac{(n-2)(n-1)}{2}$. Also from the conditions of the Theorem 5.1 there exists an element $a\notin I$ such that $b(a)\leq k$. Applying the Lemma 2.4 to $I$ and $a$, we conclude that $\dim \mathfrak{g}'\leq b(a) + \dim I'\leq\frac{(n-2)(n-1)}{2} + k$.

\end{proof}

\subsection{Proof of Theorem 1.3}

Assume the converse. Let the proper subalgebras $\mathfrak{h}_1$ and $\mathfrak{h}_2$ cover all the elements of breadth at least $n$ and $\mathfrak{h}_2$ has codimension bigger or equal to $2$ in $\mathfrak{g}$. We will prove that $\dim\mathfrak{g}'\leq n(n-1)/2$. From Lemma 2.1 the set $\mathfrak{g}\setminus(\mathfrak{h}_1\cup \mathfrak{h}_2)$ generates $\mathfrak{g}$. All the elements from $\mathfrak{g}\setminus(\mathfrak{h}_1\cup \mathfrak{h}_2)$ have breadth at most $n-1$ in $\mathfrak{g}$, thus, we can apply Theorem 5.1 to $\mathfrak{g}$, $\mathfrak{h}_1$, $\mathfrak{h}_2$, $k = n-1$ and get $\dim\mathfrak{g}'\leq \frac{(n-1)(n-2)}{2}+n-1=\frac{n(n-1)}{2}$. $\Box$

\begin{rem}

In fact, Theorem 1.2 is also a consequence of Theorem 5.1.

\end{rem}

\section{Acknowledgements}

I am grateful to Anton A. Klyachko for stating the problem.

\end{document}